\documentclass[10pt,reqno]{amsart}
  \usepackage{geometry}
\usepackage{amscd}
\usepackage{amsthm}
\usepackage{graphicx}
\usepackage{amsmath}
\usepackage{amsfonts}
\usepackage{amssymb}%
\providecommand{\U}[1]{\protect\rule{.1in}{.1in}}

\newtheorem{theorem}{Theorem}

\newtheorem{corollary}[theorem]{Corollary}

\newtheorem{definition}[theorem]{Definition}

\newtheorem{lemma}[theorem]{Lemma}

\newtheorem{proposition}[theorem]{Proposition}
\newtheorem{remark}[theorem]{Remark}

\newcommand\supp{\mathop{\rm supp}}
\begin{document}

\title[Atomic and molecular decomposition of weighted Hardy spaces]{On the atomic and molecular decomposition of weighted Hardy spaces}
\author{Pablo Rocha}
\address{Departamento de Matem\'atica, Universidad Nacional del Sur, Av. Alem 1253 - Bah\'{\i}a Blanca 8000, Buenos
Aires, Argentina.}
\email{pablo.rocha@uns.edu.ar}
\thanks{\textbf{2.010 Math. Subject Classification}: 42B15, 42B25, 42B30.}

{\it May 2023 - Corrected version of the article  \\ published in Rev. Un. Mat. Arg. (61), No. 2, (2020), 229-247}

\begin{abstract}
The purpose of this article is to give another molecular decomposition for members of weighted Hardy spaces different from that given by
Lee and Lin [J. Funct. Anal. {\bf 188} (2002), no. 2, 442–460], and to review some overlooked details.
As an application of this decomposition, we give a necessary and sufficient condition and a sufficient condition for the 
$H^{p}_{w}-L^{p}_{w}$ and the $H^{p}_{w}-H^{p}_{w}$ boundedness, respectively, of a bounded linear operator on some 
$L^{p_0}(\mathbb{R}^{n})$ with $1 < p_0 < +\infty$, for all weight $w \in \mathcal{A}_{\infty}$ and all 
$0 < p  \leq 1$ if $1< \frac{r_w -1}{r_w} p_0$ or all $0 < p < \frac{r_w -1}{r_w} p_0$ if $\frac{r_w -1}{r_w} p_0 \leq 1$, where $r_w$ is the critical index of $w$ for the reverse H\"older condition. In particular, the well known results about boundedness of singular integrals from
$H^{p}_w(\mathbb{R}^{n})$ into $L^{p}_{w}(\mathbb{R}^{n})$ and on $H^{p}_{w}(\mathbb{R}^{n})$ for all $w \in \mathcal{A}_{\infty}$ and all $0 <p \leq 1$ are established. We also obtain the $H^{p}_{w^{p}}(\mathbb{R}^{n}) - H^{q}_{w^{q}}(\mathbb{R}^{n})$ boundedness of the Riesz potential $I_{\alpha}$ for $0 < p \leq 1$, $\frac{1}{q}=\frac{1}{p} - \frac{\alpha}{n}$, and certain weights $w$. 
\end{abstract}
\maketitle

\section{Introduction}

The Hardy spaces on $\mathbb{R}^{n}$ were defined in \cite{F-S} by  C. Fefferman and E. Stein, since then the subject has received considerable attention. One of the most important applications of Hardy spaces is that they are good substitutes of Lebesgue spaces when $p \leq 1$. For example, when $p \leq 1$, it is well known that Riesz transforms are not bounded on $L^{p}(\mathbb{R}^{n})$; however, they are bounded on Hardy spaces $H^{p}(\mathbb{R}^{n})$.

To obtain the boundedness of operators, like singular integrals or fractional type operators, in the Hardy spaces $H^{p}(\mathbb{R}^{n})$, one can appeals to the atomic or molecular characterization of $H^{p}(\mathbb{R}^{n})$, which means that a distribution in $H^{p}$ can be represented as a sum of atoms or molecules. The atomic decomposition of elements in $H^{p}(\mathbb{R}^{n})$ was obtained by Coifman in \cite{coifman} ($n=1$), and by Latter in \cite{latter} ($n \geq 1$). In \cite{T-W}, Taibleson and Weiss gave the molecular decomposition of elements in $H^{p}(\mathbb{R}^{n})$. Then the boundedness of linear operators in $H^{p}$ can be deduced from their behavior on atoms or molecules in principle.
However, it must be mentioned that M. Bownik in \cite{bownik}, based on an example of Y. Meyer, constructed a linear functional defined on a dense subspace of $H^{1}(\mathbb{R}^{n})$, which maps all $(1, \infty, 0)$ atoms into bounded scalars, but yet cannot extended to a bounded linear functional on the whole $H^{1}(\mathbb{R}^{n})$. This implies that it not suffice to check that an operator from a Hardy spaces $H^{p}$, $0 < p \leq 1$, into some quasi Banach space $X$, maps atoms into bounded elements of $X$ to establish that this operator extends to  a bounded operator on $H^{p}$. Bownik's example is, in a certain sense, pathological. Fortunately, if $T$ is a classical operator, then the uniform boundedness of $T$ on atoms implies the boundedness from $H^{p}$ into $L^{p}$, this follows from the boundedness on $L^{s}$, $1 < s < \infty$, of $T$ and since one always can take an atomic decomposition which converges in the norm of $L^{s}$, see \cite{dachun} and \cite{rocha}.

The weighted Lebesgue spaces $L^{p}_{w}(\mathbb{R}^{n})$ are a generalization of the classical Lebes-gue spaces $L^{p}(\mathbb{R}^{n})$, replacing the Lebesgue measure $dx$ by the measure $w(x) dx$, where $w$ is a non-negative measurable function. Then one can define the weighted Hardy spaces $H^{p}_w(\mathbb{R}^{n})$ by generalizing the definition of $H^{p}(\mathbb{R}^{n})$ (see \cite{tor}). It is well known that the harmonic analysis on these spaces is relevant if the "weights" $w$ belong to the class $\mathcal{A}_{\infty}$.
The atomic characterization of $H^{p}_{w}(\mathbb{R}^{n})$ has been given in \cite{cuerva} and \cite{tor}. The molecular characterization of $H^{p}_{w}(\mathbb{R}^{n})$ was developed independently by X. Li and L. Peng in \cite{li} and by M.-Y. Lee and C.-C. Lin in \cite{lee}. In both works the authors obtained the boundedness of the classical singular integrals on $H^{p}_{w}$ for $w \in \mathcal{A}_1$. We extend these results for all $w \in \mathcal{A}_{\infty}$.

Given $w \in \mathcal{A}_{\infty}$, for us, a $w-(p, p_0, d)$ atom is a measurable function $a(\cdot)$ with support in a ball $B$ such that

\

$
(1) \,\, \| a \|_{L^{p_0}} \leq \displaystyle{\frac{|B|^{1/p_0}}{w(B)^{1/p}}},  \,\,\, \textit{and}
$

\

$
(2) \,\, \int x^{\alpha} a(x) \, dx =0, \,\,\,\, \textit{for all multi-index} \,\, |\alpha| \leq d,$ 

${}$ \\
where the parameters $p$, $p_0$ and $d$ satisfy certain restrictions. We remark that our definition of atom differs from that given in \cite{cuerva}, \cite{tor}.

Ones of our main results is Theorem \ref{serie atomica} of Section 2 below, this assure that:

\

\textit{If $w \in \mathcal{A}_{\infty}$ and $f $ belongs to a dense subspace of $H^{p}_w$ , then there exist a sequence of $w - (p, p_0, d)$ atoms $\{ a_j \}$ and a sequence of scalars $\{ \lambda_j \}$ with $\sum_{j} |\lambda_j |^{p} \leq c \| f \|_{H^{p}_{w}}^{p}$ such that $f = \sum_{j} \lambda_j a_j$, where the series converges to $f$ in $L^{s}(\mathbb{R}^{n})$, for all $s > 1$.}

\

With this result we avoid any problem that could arise with respect to establish the boundedness of classical operators on $H^{p}_{w}$. The verification of the convergence in $L^{s}$ for the infinite atomic decomposition was sometimes an overlooked detail. As far as the author knows, the above result has been proved for $w - (p, \infty,d)$ atoms in $\mathbb{R}$ by J. Garc\'{\i}a-Cuerva in \cite{cuerva}, and for $w-(p, \infty, d)$ atoms in $\mathbb{R}^{n}$ by D. Cruz-Uribe et al. in \cite{uribe}.

Given $w \in \mathcal{A}_{\infty}$, we say that a measurable funtion $m(\cdot)$ is a $w-(p, p_0, d)$ molecule  centered at a ball $B=B(x_0, r)$ if it satisfies the following conditions:

\

$(m1)$ $\| m \|_{L^{p_0}(B(x_0, 2r))} \leq |B|^{\frac{1}{p_0}} w(B)^{- \frac{1}{p}}$.

\

$(m2)$ $|m(x)| \leq w(B)^{-\frac{1}{p}} \left(  1 + \frac{|x-x_0|}{r} \right)^{-2n - 2d - 3}$, \, for all $x \in \mathbb{R}^{n} \setminus  B(x_0, 2r)$.

\

$(m3)$ $\int_{\mathbb{R}^{n}} x^{\alpha} m(x) dx =0$ for every multi-index $\alpha$ with $|\alpha| \leq d$.

\

Our definition of molecule is an adaptation from that given in \cite{nakai} by E. Nakai and Y. Sawano in the setting of the variable Hardy spaces. It is clear that a $w-(p, p_0, d)$ atom is a $w-(p, p_0, d)$ molecule. The pointwise inequality in $(m2)$ seems as a good substitute for "the loss of compactness in the support of an atom".

In Section 3, we obtain the following result (Theorem \ref{Hp estim mol} below):

\

\textit{ Let $0 < p \leq 1$, $w \in \mathcal{A}_{\infty}$ and $f \in \mathcal{S}'(\mathbb{R}^{n})$ such that $f = \sum_{j}  \lambda_j m_j$ in $\mathcal{S}'(\mathbb{R}^{n})$, where $\{ \lambda_j \}$ is a sequence of positive numbers  belonging to $\ell^{p}(\mathbb{N})$ and the functions $m_j$ are $(p, p_0, d)$-molecules centered at $B_j$ with respect to  the weight $w$. Then
$f \in H^{p}_{w}(\mathbb{R}^{n})$ with 
$$
\| f \|_{H^{p}_w}^{p} \leq C_{w, p, p_0}  \sum_{j} \lambda_{j}^{p}.
$$} 
With these results in Section 4 we re-establish the boundedness on $H^{p}_{w}$  and from $H^{p}_{w}$ into $L^{p}_{w}$ of certain singular integrals, for all $w \in \mathcal{A}_{\infty}$ and all $0 < p \leq 1$. We also obtain the $H^{p}_{w^{p}} - H^{q}_{w^{q}}$ boundedness of the Riesz potential $I_{\alpha}$, for $0 < p \leq 1$, $\frac{1}{q} = \frac{1}{p} - \frac{\alpha}{n}$ and certain weights $w$.

\

\textbf{Notation:} The symbol $A\lesssim B$ stands for the inequality $A\leq cB$ for some constant $c$. We denote by $B(x_0, r)$ the ball centered at $x_0 \in \mathbb{R}^{n}$ of radius $r$. Given a ball $B(x_0, r)$ and a constant $c>0$, we set $cB = B(x_0, cr)$. 
For a measurable subset $E \subset \mathbb{R}^{n}$ we denote $|E|$ and $\chi_E$ the Lebesgue measure of $E$ and the characteristic function of $E$ respectively. Given a real number $s \geq 0$, we write $\lfloor s \rfloor$ for the integer part of $s$. 
As usual we denote with $\mathcal{S}(\mathbb{R}^{n})$ the space of smooth and rapidly decreasing functions, with $\mathcal{S}'(\mathbb{R}^{n})$  the dual space. If $\beta$ is the multiindex $\beta=(\beta_1, ..., \beta_n)$,
then $|\beta| = \beta_1 + ... + \beta_n$. 

Throughout this paper, $C$ will denote a positive constant, not necessarily the same at each occurrence.

\section{Preliminaries}

\subsection{Weighted Theory}A weight is a non-negative locally integrable function on $\mathbb{R}^{n}$ that takes values in $(0, \infty)$ almost everywhere, i.e. : the weights are allowed
to be zero or infinity only on a set of Lebesgue measure zero.

Given a weight $w$ and $0 < p < \infty$, we denote by $L^{p}_{w}(\mathbb{R}^{n})$ the spaces of all functions $f$ satisfying $\| f \|_{L^{p}_{w}}^{p} := \int_{\mathbb{R}^{n}} |f(x)|^{p} w(x) dx < \infty$ . When $p=\infty$, we have that $L^{\infty}_{w}(\mathbb{R}^{n}) =L^{\infty}(\mathbb{R}^{n})$ with $\| f \|_{L^{\infty}_{w}} = \| f \|_{L^{\infty}}$. If $E$ is a measurable set, we use the notation $w(E) = \int_{E} w(x) dx$.

Let $f$ be a locally integrable function on $\mathbb{R}^{n}$. The function
$$M(f)(x) = \sup_{B \ni x} \frac{1}{|B|} \int_{B} |f(y)| dy,$$
where the supremum is taken over all balls $B$ containing $x$, is called the uncentered Hardy-Littlewood maximal function of $f$.

We say that a weight $w \in \mathcal{A}_1$ if there exists $C >  0$ such that
\begin{equation*}
M(w)(x) \leq C w(x), \,\,\,\,\, a.e. \, x \in \mathbb{R}^{n}, 
\end{equation*}
the best possible constant is denoted by $[w]_{\mathcal{A}_1}$. Equivalently, a weight $w \in \mathcal{A}_1$ if there exists $C >  0$ such that for every ball $B$
\begin{equation}
\frac{1}{|B|} \int_{B} w(x) dx \leq C \, ess\inf_{x \in B} w(x). \label{A1condequiv}
\end{equation}
\begin{remark} If $w \in \mathcal{A}_1$ and $0 < r < 1$, then by H\"older inequality we have that $w^{r} \in \mathcal{A}_1$.
\end{remark}
For $1 < p < \infty$, we say that a weight $w \in \mathcal{A}_p$ if there exists $C> 0$ such that for every ball $B$
$$\left( \frac{1}{|B|} \int_{B} w(x) dx \right) \left( \frac{1}{|B|} \int_{B} [w(x)]^{-\frac{1}{p-1}} dx \right)^{p-1} \leq C.$$
It is well known that $\mathcal{A}_{p_1} \subset \mathcal{A}_{p_2}$ for all $1 \leq p_1 < p_2 < \infty$. Also, if $w \in \mathcal{A}_{p}$ with $1 < p < \infty$, then there exists $1 < q <p$ such that $w \in \mathcal{A}_{q}$. We denote by $\widetilde{q}_{w} = \inf \{ q>1 : w \in \mathcal{A}_q \}$ \textit{the critical index of $w$.}

\begin{lemma} \label{doblante} If $w \in \mathcal{A}_{p}$ for some $1 \leq p < \infty$, then the measure $w(x) dx$ is doubling: precisely, for all $\lambda >1$ an all ball $B$ we have
\[
w(\lambda B) \leq \lambda^{np} [w]_{\mathcal{A}_p} w(B),
\]
where $\lambda B$ denotes the ball with the same center as $B$ and radius $\lambda$ times the radius of $B$.
\end{lemma}

\begin{theorem} \label{max pesada} $($\textit{Theorem 9 in \cite{Muck}}$)$ Let $1 < p < \infty$. Then
\[
\int_{\mathbb{R}^{n}} [Mf(x)]^{p} w(x) dx \leq C_{w, p, n} \int_{\mathbb{R}^{n}} |f(x)|^{p} w(x) dx,
\]
for all $f \in L^{p}_{w}(\mathbb{R}^{n})$ if and only if $w \in \mathcal{A}_p$. 
\end{theorem}

Given $1 < p \leq q < \infty$, we say that a weight $w \in \mathcal{A}_{p,q}$ if there exists $C> 0$ such that for every ball $B$
$$\left( \frac{1}{|B|} \int_{B} [w(x)]^{q} dx \right)^{1/q} \left( \frac{1}{|B|} \int_{B} [w(x)]^{-p'} dx \right)^{1/p'} \leq C < \infty.$$
For $p=1$, we say that a weight $w \in \mathcal{A}_{1,q}$ if there exists $C> 0$ such that for every ball $B$
$$\left( \frac{1}{|B|} \int_{B} [w(x)]^{q} dx \right)^{1/q} \leq C \, ess\inf_{x \in B} w(x).$$
When $p=q$, this definition is equivalent to $w^{p} \in \mathcal{A}_{p}$.
\begin{remark} \label{A_1 en A_pq}
From the inequality in (\ref{A1condequiv}) it follows that if a weight $w \in \mathcal{A}_1$, then $0 < ess\inf_{x \in B} w(x) < \infty$ for each ball $B$. Thus $w \in \mathcal{A}_1$ implies that $w^{\frac{1}{q}} \in \mathcal{A}_{p,q}$, for each $1 \leq p \leq q < \infty$.
\end{remark}
Given $0 < \alpha < n$, we define the fractional maximal operator $M_{\alpha}$ by
\[
M_{\alpha}f(x) = \sup_{B \ni x} \frac{1}{|B|^{1 - \frac{\alpha}{n}}} \int_{B} |f(y) | dy,
\]
where $f$ is a locally integrable function and the supremum is taken over all balls $B$ containing $x$

The fractional maximal operators satisfies the following weighted inequality.

\begin{theorem} \label{fract max pesada} $($\textit{Theorem 3 in \cite{Muck2}}$)$  If $0 < \alpha < n$, $1 < p < \frac{n}{\alpha}$, $\frac{1}{q} = \frac{1}{p} - \frac{\alpha}{n}$ and $w \in \mathcal{A}_{p,q}$, then
\[
\left(\int_{\mathbb{R}^{n}} [M_{\alpha}f(x)]^{q} w^{q}(x) dx\right)^{1/q} \leq C \left( \int_{\mathbb{R}^{n}} |f(x)|^{p} w^{p}(x) dx \right)^{1/p},
\]
for all $f \in L^{p}_{w^{p}}(\mathbb{R}^{n})$ . 
\end{theorem}

A weight $w$ satisfies the reverse H\"older inequality with exponent $s > 1$, denoted by $w \in RH_{s}$, if there exists $C> 0$ such that for every ball $B$,
$$\left(\frac{1}{|B|} \int_{B} [w(x)]^{s} dx \right)^{\frac{1}{s}} \leq C \frac{1}{|B|} \int_{B} w(x) dx;$$
the best possible constant is denoted by $[w]_{RH_s}$. We observe that if $w \in RH_s$, then by H\"older's inequality, $w \in RH_t$ for all $1 < t < s$, and
$[w]_{RH_t} \leq [w]_{RH_s}$. Moreover, if $w \in RH_{s}$, $s >1$, then $w \in RH_{s+ \epsilon}$ for some $\epsilon >0$. We denote by $r_w = \sup \{ r >1 : w \in RH_{r} \}$ \textit{the critical index of $w$ for the reverse H\"older condition.}

It is well known that a weight $w$ satisfies the condition $\mathcal{A}_{\infty}$ if and only if $w \in \mathcal{A}_p$ for some $p \geq 1$ (see corollary 7.3.4 in \cite{grafakos}). So $\mathcal{A}_{\infty} = \cup_{1 \leq p < \infty} \mathcal{A}_p$. Also, $w \in \mathcal{A}_{\infty}$ if and only if $w \in RH_{s}$ for some $s>1$ (see Theorem 7.3.3 in \cite{grafakos}). Thus $1 < r_w \leq +\infty$ for all $w \in \mathcal{A}_{\infty}$.

Other remarkable result about the reverse H\"older classes was discovered by Stromberg and Wheeden, they proved in \cite{wheeden} that $w \in RH_{s}$, $1 < s < + \infty$, if and only if  $w^{s} \in \mathcal{A}_{\infty}$.

Given a weight $w$, $0 < p < \infty$ and a measurable set $E$ we set $w^{p}(E) = \int_{E} [w(x)]^{p} dx$. The following result is a immediate consequence of the reverse H\"older condition.

\begin{lemma} \label{desig RH} For $0 < \alpha < n$, let $0 < p < \frac{n}{\alpha}$ and $\frac{1}{q} = \frac{1}{p} - \frac{\alpha}{n}$. If $w^{p} \in RH_{\frac{q}{p}}$ then
\[
[w^{p}(B)]^{-\frac{1}{p}} [w^{q}(B)]^{\frac{1}{q}} \leq [w^{p}]_{RH_{q/p}}^{1/p} |B|^{-\frac{\alpha}{n}},
\]
for each ball $B$ in $\mathbb{R}^{n}$.
\end{lemma}

\subsection{Weighted Hardy Spaces} Topologize $\mathcal{S}(\mathbb{R}^{n})$ by the collection of semi-norms $\| \cdot \|_{\alpha, \beta}$, with $\alpha$ and $\beta$ multi-indices, given by
$$\| \varphi \|_{\alpha, \beta} = \sup_{x \in \mathbb{R}^{n}} |x^{\alpha} \partial^{\beta}\varphi(x)|.$$
For each $N \in \mathbb{N}$, we set $\mathcal{S}_{N}=\left\{  \varphi\in \mathcal{S}(\mathbb{R}^{n}): \| \varphi \|_{\alpha, \beta} \leq 1, |\alpha|, |\beta| \leq N \right\}$. Let $f \in \mathcal{S}'(\mathbb{R}^{n})$, we denote by $\mathcal{M}_{N}$ the grand maximal
operator given by
\[
\mathcal{M}_{N}f(x)=\sup\limits_{t>0}\sup\limits_{\varphi\in\mathcal{S}_{N}
}\left\vert \left(  t^{-n}\varphi(t^{-1} \cdot)\ast f\right)  \left(  x\right)
\right\vert.
\]
Given a weight $w \in \mathcal{A}_{\infty}$ and $p > 0$, the weighted Hardy space $H^{p}_w(\mathbb{R}^{n})$ consists of all tempered distributions $f$ such that
\[
\| f \|_{H^{p}_w(\mathbb{R}^{n})} = \| \mathcal{M}_{N}f \|_{L^{p}_w(\mathbb{R}^{n})} = \left( \int_{\mathbb{R}^{n}}  [\mathcal{M}_{N}f(x)]^{p} w(x) dx \right)^{1/p} < \infty.
\]

Let $\phi \in \mathcal{S}(\mathbb{R}^{n})$ be a function such that $\int \phi(x) dx \neq 0$. For $f \in \mathcal{S}'(\mathbb{R}^{n})$, we define the maximal function $M_{\phi}f$ by
$$M_{\phi}f(x)= \sup_{t>0} \left\vert \left(  t^{-n}\phi(t^{-1} \, \cdot)\ast f\right)  \left(  x\right)\right\vert.$$
For $N$ sufficiently large, we have  $\| M_{\phi}f \|_{L^{p}_w} \simeq \| \mathcal{M}_{N}f \|_{L^{p}_w}$, (see \cite{tor}).

\

In the sequel we consider the following set
\[
\widehat{\mathcal{D}}_{0} = \left\{ \phi \in \mathcal{S}(\mathbb{R}^{n}) : \widehat{\phi} \in C_{c}^{\infty}(\mathbb{R}^{n}) \,\,
\textit{and} \,\,\, 0 \notin \supp ( \widehat{\phi} \, ) \right\}.
\]
The following theorem is crucial to get the main results.

\begin{theorem} \label{dense set} $($Theorem 1 in \cite{tor} pp. $103$$)$ Let $w$ be a doubling weight on $\mathbb{R}^{n}$. Then $\widehat{\mathcal{D}}_{0}$ is dense in $H^{p}_{w}(\mathbb{R}^{n})$, $0 < p  < \infty$.
\end{theorem}

\subsubsection{Weighted atom Theory} Let $w \in \mathcal{A}_{\infty}$ with critical index $\widetilde{q}_w$ and critical index $r_w$ for reverse H\"older condition. Let $0 < p \leq 1$, $\max \{ 1, p(\frac{r_w}{r_w -1}) \} < p_0 \leq +\infty$,
and $d \in \mathbb{Z}$ such that $d \geq \lfloor n(\frac{\widetilde{q}_w}{p} -1) \rfloor$, we say that a function $a(\cdot)$ is a $w - (p, p_0, d)$ atom centered in $x_0 \in \mathbb{R}^{n}$ if

\

$(a1)$ $a \in L^{p_0}(\mathbb{R}^{n})$ with support in the ball  $B= B(x_0, r)$.

\

$(a2)$ $\| a \|_{L^{p_0}(\mathbb{R}^{n})} \leq |B|^{\frac{1}{p_0}}w(B)^{-\frac{1}{p}}$.

\

$(a3)$ $\int x^{\alpha} a(x) dx  =0$ for all multi-index $\alpha$ such that $| \alpha | \leq d$.

\

We observe that the condition $\max \{ 1, p(\frac{r_w}{r_w -1}) \} < p_0 < +\infty$ implies that $w \in RH_{(\frac{p_0}{p})'}$. If $r_w = +\infty$, then $w \in RH_{t}$ for each $1 < t < +\infty$. So, if $r_w = + \infty$ and since $\displaystyle{\lim_{t \rightarrow +\infty}} \frac{t}{t-1} = 1$ we put $\frac{r_w}{r_w - 1} =1$. For example, if $w \equiv 1$, then $\widetilde{q}_w =1$ and $r_w = + \infty$ and our definition of atom in this case coincide with the definition of atom in the classical Hardy spaces.

\begin{lemma} \label{estim atomo} Let $w \in \mathcal{A}_{\infty}$ with critical index $\widetilde{q}_{w}$ and critical index $r_w$. If $a(\cdot)$ is a $w - (p, p_0, d)$ atom, then $a(\cdot) \in H^{p}_{w}(\mathbb{R}^{n})$. Moreover, there exists a positive constant $C$ independent of the atom $a$ such that $\| a \|_{ H^{p}_{w}} \leq C$.
\end{lemma}

\begin{proof} Let $\phi \in \mathcal{S}(\mathbb{R}^{n})$ with $\int \phi(x) dx \neq 0$. Since $\phi$ has a radial majorant that is a non-incresing, bounded and integrable function, we have that
\[
M_{\phi}a(x) \leq c Ma(x), \,\,\,\,\,\, for \,\, all \,\,\,\, x \in \mathbb{R}^{n}.
\]
In view of the moment condition of $a$ we have
\[
(a \ast \phi_t )(x) = \int [\phi_{t}(x-y) - q_{x,t}(y)] a(y) dy, \,\,\,\, if \,\,\, x \in \mathbb{R}^{n} \setminus B(x_0, 4r)
\]
where $q_{x,t}$ is the degree $d$ Taylor polynomial of the function $y \rightarrow \phi_{t}(x-y)$ expanded around $x_0$. By the standard estimate of the remainded term of the Taylor expansion, the condition $(a2)$ and H\"older's inequality, we obtain that
 \begin{eqnarray*}
M_{\phi}a(x) &\leq& c \| a \|_1 r^{d+1} |x-x_0|^{-n-d-1} \\
&\leq& c w(B)^{-1/p} r^{n+d+1} |x-x_0|^{-n-d-1} \\ 
&\leq& c w(B)^{-1/p} [M(\chi_{B})(x)]^{\frac{n+d+1}{n}}, \,\,\, if \,\, x \in \mathbb{R}^{n} \setminus B(x_0, 4r).
\end{eqnarray*}
Therefore
\[
\int [M_{\phi}a(x)]^{p} w(x) dx \leq c \int \left( \chi_{B(x_0, 4r)} [Ma(x)]^{p} +  \frac{ [M(\chi_{B})(x)]^{\frac{(n+d+1)p}{n}}}{w(B)} \right) w(x) dx.
\]
On the right-side of this inequality, we apply  H\"older's inequality with $p_0 /p$ and use that $w \in RH_{(\frac{p_0}{p})'}$ ($p_0 > p(\frac{r_w}{r_w -1})$) and Lemma \ref{doblante} for the first term, for the second term we have that $\frac{(n+d+1)p}{n} > \widetilde{q}_w$, so $w \in \mathcal{A}_{\frac{(n+d+1)p}{n}}$, then to invoke Theorem \ref{max pesada} we obtain
\[
\| M_{\phi}a\|_{L^{p}_{w}}^{p} = \int_{\mathbb{R}^{n}} [M_{\phi}a(x)]^{p} w(x) dx \leq C,
\]
where the constant $C$ is independent of the $w - (p, p_0, d)$ atom $a$. Thus $a \in H^{p}_{w}(\mathbb{R}^{n})$.
\end{proof}

\begin{theorem} \label{serie atomica} Let $f \in \widehat{\mathcal{D}}_{0}$, and $0 < p \leq 1$. If $w \in \mathcal{A}_{\infty}$, then there exist a sequence of $w - (p, p_0, d)$ atoms $\{ a_j \}$ and a sequence of scalars $\{ \lambda_j \}$ with $\sum_{j} |\lambda_j |^{p} \leq c \| f \|_{H^{p}_{w}}^{p}$ such that $f = \sum_{j} \lambda_j a_j$, where the convergence is both in $L^{s}(\mathbb{R}^{n})$ and pointwise, for each $1 < s <\infty$.
\end{theorem}

\begin{proof} Given $f \in \widehat{\mathcal{D}}_{0}$, let $\mathcal{O}_j = \{ x : \mathcal{M}_N f(x) > 2^{j} \}$ and let 
$\mathcal{F}_j = \{ Q_{k}^{j} \}_{k}$ be the Whitney decomposition associated to $\mathcal{O}_j$ such that 
$\bigcup_{k} Q_{k}^{j \ast} = \mathcal{O}_j$, where $Q_{k}^{j \ast}$ is the cube with the same center as $Q_{k}^{j}$ but with side length
$(1+\epsilon)$ times that of $Q_{k}^{j}$ with $0 < \epsilon < \frac{1}{4}$ suitably chosen (see \cite[Appendix J]{grafakos}). Fixed $j \in \mathbb{Z}$, we define the following set
$$E^{j} = \{ (i,k) \in \mathbb{Z} \times \mathbb{Z} : Q_{i}^{j+1 \ast} \cap Q_{k}^{j \ast} \neq \emptyset \}$$
and let $E_{k}^{j} = \{ i : (i,k) \in E^{j} \}$ and $E_{i}^{j} = \{ k : (i,k) \in E^{j} \}$.
Following the proof of \cite[Theorem 2, pp. 107-109]{stein}, we have a sequence of functions $A_{k}^{j}$ such that

\qquad

(i) $supp(A_{k}^{j}) \subset Q_{k}^{j \ast} \cup \bigcup_{i \in E_{k}^{j}} Q_{i}^{j+1 \ast}$ and
$|A_{k}^{j}(x)| \leq c 2^{j}$ for all $k,j \in \mathbb{Z}$.

\qquad

(ii) $\int x^{\alpha} A_{k}^{j}(x) dx =0$ for all $\alpha$ with $|\alpha| \leq d$ and all $k,j \in \mathbb{Z}$.

\qquad

(iii) The sum $\sum_{j,k} A_{k}^{j}$ converges to $f$ in the sense of distributions. \\
From (i) we obtain
$$\sum_{k} |A_{k}^{j}| \leq c 2^{j} \left( \sum_{k} \chi_{Q_{k}^{j \ast}} + \sum_{k} \chi_{\bigcup_{i \in E^{j}_{k}} Q_{i}^{j+1 \ast}} \right)$$
following the proof of Theorem 3.1 in \cite{rocha} we obtain that
$$\leq c 2^{j} \left( \chi_{\mathcal{O}_j} + \sum_{k} \sum_{i \in E^{j}_{k}} \chi_{Q_{i}^{j+1 \ast}} \right)
= c 2^{j} \left( \chi_{\mathcal{O}_j} + \sum_{i} \sum_{k \in E^{j}_{i}} \chi_{Q_{i}^{j+1 \ast}} \right)$$
$$\leq c 2^{j} \left( \chi_{\mathcal{O}_j} + 84^{n} \sum_{i} \chi_{Q_{i}^{j+1 \ast}} \right)
\leq c 2^{j} \left( \chi_{\mathcal{O}_j} + \chi_{\mathcal{O}_{j+1}} \right) \leq c 2^{j} \chi_{\mathcal{O}_j},$$
by Lemma 2.4 in \cite{rocha} we have that
$$\sum_{j,k} |A_{k}^{j}(x)| \leq c \sum_{j} 2^{j} \chi_{\mathcal{O}_j \setminus \mathcal{O}_{j+1}}(x), \,\,\,\, a.e. x \in \mathbb{R}^{n}.$$
Thus, for $1 < s < \infty$ fixed
\begin{equation} 
\int \left( \sum_{j,k} |A_{k}^{j}(x)| \right)^{s} dx \leq c \sum_j \int_{\mathcal{O}_j \setminus \mathcal{O}_{j+1}} 2^{js} dx \leq c \sum_j \int_{\mathcal{O}_j \setminus \mathcal{O}_{j+1}} (\mathcal{M}_N f(x))^{s} dx \label{sum}
\end{equation}
$$\leq c \int_{\mathbb{R}^{n}} (\mathcal{M}_N f(x))^{s}  dx< \infty,$$
since $f \in \hat{\mathcal{D}}_{0} \subset L^{s}(\mathbb{R}^{n})$. From (\ref{sum}) and (iii) we obtain that the sum $\sum_{j,k} A_{k}^{j}$ converges to $f$ in $L^{s}(\mathbb{R}^{n})$, and $\sum_{j,k} A_{k}^{j}(x) =f(x)$ a.e.$x \in \mathbb{R}^{n}$, for each $1 < s < \infty$.

Now, we set $a_{j,k} = \lambda^{-1}_{j,k} A_{k}^{j}$ with $\lambda_{j,k} = c 2^{j} w(B_{k}^{j})^{1/p}$, where $B_{k}^{j}$ is the smallest ball containing $Q_{k}^{j \ast}$ as well as all the $Q_{i}^{j+1 \ast}$ that intersect $Q_{k}^{j \ast}$. Then we have a sequence $\{ a_{j,k} \}$ of $w - (p, p_0, d)$ atoms and a sequence of scalars $\{ \lambda_{j,k} \}$ such that the sum $\sum_{j,k} \lambda_{j,k} a_{j,k}$ converges to $f$ in $L^{s}(\mathbb{R}^{n})$ and a.e.$x \in \mathbb{R}^{n}$. On the other hand there exists an universal constant $c_1$ such that $B_{k}^{j} \subset c_1 Q_{k}^{j}$. Then, by Lemma \ref{doblante}
\[
\sum_{j,k} |\lambda_{j,k} |^{p} \lesssim  \sum_{j,k} 2^{jp} w(B_{k}^{j}) \lesssim  \sum_{j,k} 2^{jp} w(c_1 Q_{k}^{j}) \lesssim c_1^{np} \sum_{j,k} 2^{jp} w(Q_{k}^{j}) \leq c \sum_{j} 2^{jp} w(\mathcal{O}_j ).
\]
If $x \in \mathbb{R}^{n}$, there exists a unique $j_0 \in \mathbb{Z}$ such that $2^{j_0 p} <  \mathcal{M}_N f(x)^{p} \leq 2^{(j_0 +1)p}$. So
\[
\sum_{j} 2^{jp} \chi_{\mathcal{O}_j }(x) \leq \sum_{j \leq j_0} 2^{j p} =\frac{ 2^{(j_0+1)p} }{2^{p}-1} \leq\frac{ 2^{p}}{2^{p}-1} \mathcal{M}_N f(x)^{p}.
\]
From this it follows that
\[
\sum_{j,k} |\lambda_{j,k} |^{p} \leq c \sum_{j} 2^{jp} w(\mathcal{O}_j ) \leq c \frac{ 2^{p}}{2^{p}-1} \| \mathcal{M}_N f \|_{L^{p}_{w}}^{p} = c \frac{ 2^{p}}{2^{p}-1} \| f \|_{H^{p}_{w}}^{p},
\]
which proves  the theorem.
\end{proof}

\begin{theorem} \label{estim Hpw-Lpw}
Let $T$ be a bounded linear operator from $L^{p_0}(\mathbb{R}^{n})$ into $L^{p_0}(\mathbb{R}^{n})$ for some $1 < p_0 < +\infty$. If $w \in \mathcal{A}_{\infty}$ with critical index $r_w$,
$0 < p \leq 1 < \frac{r_w -1}{r_w} p_0$ or  $0 < p < \frac{r_w -1}{r_w} p_0 \leq 1$, then $T$ can be extended to an $H^{p}_{w}(\mathbb{R}^{n}) - L^{p}_{w}(\mathbb{R}^{n})$ bounded linear operator if and only if  $T$  is uniformly bounded into $L^{p}_{w}$ norm on all $w-(p, p_0, d)$ atom $a$.
\end{theorem}

\begin{proof}
 Since $T$ is a bounded linear operator on $L^{p_0}(\mathbb{R}^{n})$, $T$ is well defined on $H^{p}_w(\mathbb{R}^{n}) \cap L^{p_0}(\mathbb{R}^{n})$. If $T$ can be extended to a bounded operator from $H^{p}_w(\mathbb{R}^{n})$ into $L^{p}_w(\mathbb{R}^{n})$, then $\| Ta \|_{L^{p}_w} \leq c_p \|a \|_{H^{p}_w}$ for all $w$-atom $a$. By Lemma \ref{estim atomo}, there exists a universal constant $C$ such that $\|a \|_{H^{p}_w} \leq C < \infty$ for all $w$-atom $a$; so $\| Ta \|_{L^{p}_w} \leq C_p $ for all $w$-atom $a$. 

Conversely, taking into account the assumptions on $p$ and $p_0$, given $f \in \widehat{\mathcal{D}}_0$, by Theorem \ref{serie atomica}, there exists a $w-(p, p_0, d)$ atomic decomposition such that $\sum_j |\lambda_j|^{p} \lesssim \| f \|_{H^{p}_w}$ and $\sum_{j} \lambda_j a_j = f$ in $L^{p_0}(\mathbb{R}^{n})$. From the boundedness of $T$ on $L^{p_0}(\mathbb{R}^{n})$ we have that the sum $\sum_j \lambda_j Ta_j$ converges a $Tf$ in $L^{p_0}(\mathbb{R}^{n})$, thus there exists a subsequence of natural numbers $\{k_N \}_{N \in \mathbb{N}}$ such that $\lim_{N \rightarrow +\infty} \sum_{j=-k_N}^{k_N} \lambda_{j} Ta_{j}(x)= Tf(x)$ a.e.$x \in \mathbb{R}^{n}$, this implies
$$|Tf(x)| \leq \sum_j |\lambda_j Ta_j(x)|, \,\,\,\, a.e.x \in \mathbb{R}^{n}.$$
If $\| Ta \|_{L^{p}_w} \leq C_p < \infty $ for all $w-(p, p_0, d)$ atom $a$, and since $0< p \leq 1$ we get
$$\|Tf \|_{L^{p}_w}^{p} \leq \sum_j |\lambda_j|^{p} \|Ta_j \|_{L^{p}_w}^{p} \leq C_{p}^{p} \sum_j |\lambda_j|^{p} \leq C_{p}^{p} \|f \|_{H^{p}_w}^{p}$$
for all $f \in\widehat{\mathcal{D}}_0$. By theorem \ref{dense set} and Lemma \ref{doblante}, we have that $\widehat{\mathcal{D}}_0$ is a dense subspace of $H^{p}_w(\mathbb{R}^{n})$, so the theorem follows by a density argument.
\end{proof}

\section{Molecular decomposition}

Before giving our molecular reconstruction for $H^{p}_{w}(\mathbb{R}^{n})$, we need to introduce the following discrete maximal: 
given $\phi \in \mathcal{S}(\mathbb{R}^{n})$ and $f \in \mathcal{S}'(\mathbb{R}^{n})$, we define
$$M^{dis}_{\phi}f(x) = \sup_{j \in \mathbb{Z}} \left| (\phi^{j} \ast f) (x) \right|,$$
where $\phi^{j}(x) = 2^{jn} \phi(2^{j}x).$ From \cite[Lemma 3.2 and Proof of Theorem 3.3]{nakai}, it follows that for all 
$f \in  \mathcal{S}'(\mathbb{R}^{n})$ and all $0 < \theta < 1$
\begin{equation} \label{maximal discreta}
\mathcal{M}_N f(x) \leq C \left[ M \left(\left(M^{dis}_{\phi}f \right)^{\theta}\right)(x)\right]^{\frac{1}{\theta}}, \,\,\, for \,\, all \,\, x \in \mathbb{R}^{n}, 
\end{equation}
if $N$ is sufficiently large. This inequality gives the following result.

\begin{lemma} \label{Hpw norm discreta} Let $\phi \in \mathcal{S}(\mathbb{R}^{n})$ and $f \in \mathcal{S}'(\mathbb{R}^{n})$.
If $\omega \in \mathcal{A}_{\infty}$ and $0 < p \leq 1$, then $\| f \|_{H^{p}_{w}} \leq C \| M^{dis}_{\phi}f \|_{L^{p}_{w}}$,
where $C$ is a positive constant which does not depend on $f$.
\end{lemma}

\begin{proof} From (\ref{maximal discreta}) above, we have
\[
\| f \|_{H^{p}_{w}} = \| \mathcal{M}_N f \|_{L^{p}_{w}} \lesssim 
\left\| \left[ M \left(\left(M^{dis}_{\phi}f \right)^{\theta}\right)(\cdot)\right]^{\frac{1}{\theta}} \right\|_{L^{p}_{w}}
= \left\| M \left(\left(M^{dis}_{\phi}f \right)^{\theta}\right) \right\|_{L^{p/\theta}_{w}}^{1/\theta}.
\] 
By taking $0 < \theta < p$ such that $\frac{p}{\theta} > \widetilde{q}_w$, by Theorem \ref{max pesada}, we get
\[
\| f \|_{H^{p}_{w}} \lesssim
\left\| M \left(\left(M^{dis}_{\phi}f \right)^{\theta}\right) \right\|_{L^{p/\theta}_{w}}^{1/\theta}
\lesssim \left\| \left(M^{dis}_{\phi}f \right)^{\theta} \right\|_{L^{p/\theta}_{w}}^{1/\theta}
= \left\| M^{dis}_{\phi}f \right\|_{L^{p}_{w}}.
\]
Consequently, we obtain the desired result.
\end{proof}

Our definition of molecule is an adaptation from that given in \cite{nakai} by E. Nakai and Y. Sawano in the setting of the variable Hardy spaces.

\begin{definition} Let $w \in \mathcal{A}_{\infty}$ with critical index $\widetilde{q}_w$ and critical index $r_w$ for reverse H\"older condition. Let $0 < p \leq 1$, $\max \{ 1, p(\frac{r_w}{r_w -1}) \} < p_0 \leq +\infty$ and $d \in \mathbb{Z}$ such that $d \geq \lfloor n(\frac{\widetilde{q}_w}{p} -1) \rfloor$. We say that a function $m(\cdot)$ is a $w-(p,p_0, d)$ molecule centered at a ball $B=B(x_0, r)$ if it satisfies the following conditions:

\

$(m1)$ $\| m \|_{L^{p_0}(B(x_0, 2r))} \leq |B|^{\frac{1}{p_0}} w(B)^{- \frac{1}{p}}$.

\

$(m2)$ $|m(x)| \leq w(B)^{-\frac{1}{p}} \left(  1 + \frac{|x-x_0|}{r} \right)^{-2n - 2d - 3}$ for all $x \in \mathbb{R}^{n} \setminus  B(x_0, 2r)$.

\

$(m3)$ $\int_{\mathbb{R}^{n}} x^{\alpha} m(x) dx =0$ for every multi-index $\alpha$ with $|\alpha| \leq d$.
\end{definition}

\begin{remark} \label{Lp estim molecule}
The conditions $(m1)$ and $(m2)$ imply that $\| m \|_{L^{p_0}(\mathbb{R}^{n})} \leq c \frac{|B|^{\frac{1}{p_0}}}{ w(B)^{\frac{1}{p}}}$, where $c$ is a positive constant independent of the molecule $m$.
\end{remark}

From the definition of molecule is clear that a $w-(p, p_0, d)$ atom is a $w-(p, p_0, d)$ molecule.

In view of Lemma \ref{estim atomo}, the following theorem assure, among other things, that the pointwise inequality in $(m2)$ is a good substitute for "the loss of compactness in the support of an atom".

\begin{theorem} \label{Hp estim mol} Let $0 < p \leq 1$, $w \in \mathcal{A}_{\infty}$ and $f \in \mathcal{S}'(\mathbb{R}^{n})$ such that $f = \sum_{j}  \lambda_j m_j$ in $\mathcal{S}'(\mathbb{R}^{n})$, where $\{ \lambda_j \}$ is a sequence of positive numbers  belonging to $\ell^{p}(\mathbb{N})$ and the functions $m_j$ are $(p, p_0, d)$-molecules centered at $B_j$ with respect to  the weight $w$. Then
$f \in H^{p}_{w}(\mathbb{R}^{n})$ with 
\[
\| f \|_{H^{p}_w}^{p} \leq C_{w, p, p_0}  \sum_{j} \lambda_{j}^{p}.
\] 
\end{theorem}

\begin{proof} Let $\phi \in C_{c}^{\infty}(\mathbb{R}^{n})$ such that $\chi_{B(0,1)} \leq \phi  \leq \chi_{B(0,2)}$, we set $\phi^{k}(x) = 2^{kn}\phi(2^{k}x)$ where $k \in \mathbb{Z}$. Since  $f = \sum_{j}  \lambda_j m_j$ in the sense of the distributions we have that 
\[
|(\phi^{k} \ast f)(x)| \leq \sum_{j=1}^{\infty} \lambda_j |(\phi^{k} \ast m_j)(x)|,
\]
for all $x \in \mathbb{R}^{n}$ and all $k \in \mathbb{Z}$. We observe that the argument used in the proof of Theorem 5.2 in \cite{nakai} to obtain the pointwise inequality $(5.2)$ works in this setting, but considering now the conditions $(m1)$, $(m2)$ and $(m3)$. Therefore we get
\[
M_{\phi}^{dis}(f)(x) \lesssim \sum_{j} \lambda_{j} \chi_{2 B_j}(x) M(m_j)(x) +   \sum_{j} \lambda_{j} \frac{ \left[ M(\chi_{B_j})(x) \right]^{\frac{n + d_w +1}{n}}}{w(B_j)^{\frac{1}{p}}}, \,\,\,\, (x \in \mathbb{R}^{n})
\]
where $M$ is the Hardy-Littlewood maximal operator.

Since $0 < p \leq 1$, it follows that
\[
[M_{\phi}^{dis}(f)(x)]^{p} \lesssim \sum_{j} \lambda_{j}^{p} \chi_{2 B_j}(x) [M(m_j)(x)]^{p} +   \sum_{j} \lambda_{j}^{p} \frac{ \left[ M(\chi_{B_j})(x) \right]^{p \frac{n + d_w +1}{n}}}{w(B_j)}, \,\,\,\, (x \in \mathbb{R}^{n})
\]
by integrating with respect to $w$ we get 
\[
\int [M_{\phi}^{dis}(f)(x)]^{p} w(x) dx \lesssim \sum_{j} \lambda_{j}^{p} \int \chi_{2 B_j}(x) [M(m_j)(x)]^{p} w(x) dx 
\]
\[
+   \sum_{j} \lambda_{j}^{p} \int \frac{ \left[ M(\chi_{B_j})(x) \right]^{p \frac{n + d_w +1}{n}}}{w(B_j)} w(x) dx.
\]
On the right-side of this inequality, we apply  H\"older's inequality with $p_0 /p$, Remark \ref{Lp estim molecule}, Lemma \ref{doblante} and use that $w \in RH_{(\frac{p_0}{p})'}$ ($p_0 > p(\frac{r_w}{r_w -1})$) for the first term, for the second term we have that $\frac{(n+d+1)p}{n} >\widetilde{q}_w$, so $w \in \mathcal{A}_{\frac{(n+d+1)p}{n}}$, then we apply Theorem \ref{max pesada}. Finally, by invoking Lemma \ref{Hpw norm discreta} we obtain
\[
\| f \|_{H^{p}_w}^{p} \leq C_{w, p, p_0}  \sum_{j} \lambda_{j}^{p}.
\] 
This completes the proof.
\end{proof}

\begin{theorem} \label{Hpw-Hpw estim}
Let $T$ be a bounded linear operator from $L^{p_0}(\mathbb{R}^{n})$ into $L^{p_0}(\mathbb{R}^{n})$ for some $1 < p_0 < +\infty$. If $w \in \mathcal{A}_{\infty}$ with critical index $r_w$,
$0 < p \leq 1 < \frac{r_w -1}{r_w} p_0$ or  $0 < p < \frac{r_w -1}{r_w} p_0 \leq 1$ and $Ta$ is a $w - (p, p_0, d_2)$ molecule for each 
$w - (p, p_0, d_1)$ atom $a$, then $T$ can be extended to an $H^{p}_{w}(\mathbb{R}^{n}) - H^{p}_{w}(\mathbb{R}^{n})$ bounded linear operator.
\end{theorem}

\begin{proof} Taking into account the assumptions on $p$ and $p_0$,
given $f \in \widehat{\mathcal{D}}_0$, from Theorem \ref{serie atomica} it follows that there exists a sequence of $w - (p, p_0, d_1)$ atoms $\{ a_j \}$ and a sequence of scalars $\{ \lambda_j \}$ with 
\begin{equation}
\sum_{j} |\lambda_j |^{p} \lesssim \| f \|_{H^{p}_{w}}^{p},  \label{atomic decomp}
\end{equation}
such that $f = \sum_{j} \lambda_j a_j$ in $L^{p_0}(\mathbb{R}^{n})$. From the boundedness of $T$ on $L^{p_0}(\mathbb{R}^{n})$ we have that $Tf = \sum_{j} \lambda_j Ta_j$ in $L^{p_0}(\mathbb{R}^{n})$ and therefore in $\mathcal{S}'(\mathbb{R}^{n})$. By hypothesis $Ta_j$ is a $w - (p, p_0, d_2)$ molecule for all $j$, so Theorem \ref{Hp estim mol} and the inequality in (\ref{atomic decomp}) imply that
\[
\| Tf \|_{H^{p}_w}^{p} \lesssim \sum_{j} |\lambda_j |^{p} \lesssim \| f \|_{H^{p}_{w}}^{p}, 
\]
for all $f \in \widehat{\mathcal{D}}_0$, so the theorem follows from the density of  $\widehat{\mathcal{D}}_0$ in $H^{p}_{w}(\mathbb{R}^{n})$.
\end{proof}

\section{Applications}

\subsection{Singular integrals} Let $\Omega \in C^{\infty}(S^{n-1})$ with $\int_{S^{n-1}} \Omega(u) d\sigma(u)=0$. We define the operator $T$ by
\begin{equation}
Tf(x) = \lim_{\epsilon \rightarrow 0^{+}} \int_{|y| > \epsilon} \frac{\Omega(y/|y|)}{|y|^{n}} f(x-y) \, dy, \,\,\,\, x \in \mathbb{R}^{n}. \label{T}
\end{equation}
It is well known that $\widehat{Tf}(\xi) = m(\xi) \widehat{f}(\xi)$, where the multiplier $m$ is homogeneous of degree $0$ and is indefinitely diferentiable on $\mathbb{R}^{n} \setminus \{0\}$. Moreover, if $k(y) = \frac{\Omega(y/|y|)}{|y|^{n}}$ we have
\begin{equation}
|\partial^{\alpha}_{y} k(y) |\leq C |y|^{-n-|\alpha|}, \,\,\,\,  \textit{for all } \,\, y \neq 0 \,\, \textit{and all multi-index} \,\, \alpha. \label{k estimate}
\end{equation}
The operator $T$ results bounded on $L^{s}(\mathbb{R}^{n})$ for all $1 < s < +\infty$ and of weak-type $(1,1)$ (see \cite{elias}).

Let $0 < p \leq 1$ and $d = \lfloor n (\frac{\widetilde{q}_w}{p}-1) \rfloor$. Given a $w-(p, p_0, n+2d+2)$ atom $a(\cdot)$ with support in the ball $B(x_0, r)$ we have that
\begin{equation}
\| Ta \|_{L^{p_0}(B(x_0, 2r))} \leq C \| a \|_{p_0} \leq C |B|^{1/p_0} w(B)^{-1/p}, \label{Ta 0}
\end{equation}
since $T$ is bounded on $L^{p_0}(\mathbb{R}^{n})$. In view of the moment condition of $a(\cdot)$ we obtain
\[
Ta(x)=  \int_{B} k(x-y) a(y) dy = \int_{B} [k(x-y) - q_{n+2d+2}(x,y)] a(y) dy, \,\, x \notin B=B(x_0, 2r)
\]
where $q_{n+2d+2}$ is the degree $n+2d+2$ Taylor polynomial of the function $y \rightarrow k(x-y)$ expanded around $x_0$. From the estimate in (\ref{k estimate}), and the standard estimate of the remainder term of the Taylor expansion, there exists $\xi$ between  $y$ and $x_0$ such that
\begin{equation}
| Ta(x) | \leq C \| a \|_{1} \frac{r^{n+2d+3}}{|x - \xi|^{2n+2d+3}}
\leq C\frac{ r^{2n+2d+3}}{w(B)^{1/p}}  |x - x_0|^{-2n-2d-3}, \,\,\, x \notin B(x_0, 2r), \label{Ta}
\end{equation}
this inequality and a simple computation allow us to obtain
\begin{equation}
|Ta(x)| \leq C  w(B)^{-\frac{1}{p}} \left(  1 + \frac{|x-x_0|}{r} \right)^{-2n - 2d - 3}, \,\,\, \textit{for all} \,\, x \notin  B(x_0, 2r).  \label{Ta 2}
\end{equation}
From the estimate in (\ref{Ta}) we obtain that the function $x \rightarrow x^{\alpha} Ta(x)$ belongs to $L^{1}(\mathbb{R}^{n})$ for each $|\alpha| \leq d$, so
\[
|((-2\pi i x)^{\alpha}Ta) \,\, \widehat{} \,\, (\xi) | = |\partial^{\alpha}_{\xi} (m(\xi) \widehat{a}(\xi))| = \left| \sum_{\beta \leq \alpha} c_{\alpha, \beta} \, (\partial^{\alpha - \beta}_{\xi} m)(\xi) \,  (\partial^{\beta}_{\xi} \widehat{a}) (\xi) \right|
\]
\[
= \left| \sum_{\beta \leq \alpha} c_{\alpha, \beta} \, (\partial^{\alpha - \beta}_{\xi} m)(\xi) \,  ((-2\pi i x)^{\beta} a) \,\, \widehat{} \,\, (\xi) \right|,
\]
from the homogeneity of the function $\partial^{\alpha - \beta}_{\xi} m$ we obtain that
\begin{equation}
|((-2\pi i x)^{\alpha}Ta) \,\, \widehat{} \,\, (\xi) |\leq C \sum_{\beta \leq \alpha}| c_{\alpha, \beta}|  \frac{\left| ((-2\pi i x)^{\beta} a) \,\, \widehat{} \,\, (\xi) \right|}{|\xi|^{|\alpha| -| \beta|}}, \,\,\,  \xi \neq 0. \label{limite}
\end{equation}
Since $\displaystyle{\lim_{\xi \rightarrow 0}}  \frac{\left| ((-2\pi i x)^{\beta} a) \,\, \widehat{} \,\, (\xi) \right|}{|\xi|^{|\alpha| -| \beta|}} =0$ for each $\beta \leq \alpha$ (see 5.4, pp. 128, in \cite{stein}), taking the limit as $\xi \rightarrow 0$ at (\ref{limite}), we get 
\begin{equation}
\int_{\mathbb{R}^{n}} (-2\pi i x)^{\alpha} Ta(x) \, dx = ((-2\pi i x)^{\alpha}Ta) \,\, \widehat{} \,\, (0) =0, \,\,\, \textit{for all} \,\,\, |\alpha| \leq d. \label{Ta 3} 
\end{equation}
From (\ref{Ta 0}), (\ref{Ta 2}) and (\ref{Ta 3}) it follows that there exists an universal constant $C>0$ such that $C Ta(\cdot)$ is a $w-(p, p_0, d)$ molecule if $a(\cdot)$ is a $w-(p, p_0, n+2d+2)$ atom. Taking $p_0 \in (1, +\infty)$ such that $1 < \frac{r_w -1}{r_w}p_0$ and since $T$ is bounded on $L^{p_0}(\mathbb{R}^{n})$, by Theorem \ref{Hpw-Hpw estim}, we get the following result.

\begin{theorem}
Let $T$ be the operator defined in (\ref{T}). If $w \in \mathcal{A}_{\infty}$ and $0 < p \leq 1$, then $T$ can be extended to an $H^{p}_{w}(\mathbb{R}^{n}) - H^{p}_{w}(\mathbb{R}^{n})$ bounded operator.
\end{theorem}

In particular, the Hilbert transform and the Riesz transforms admit a continuous extension on $H^{p}_{w}(\mathbb{R})$ and $H^{p}_{w}(\mathbb{R}^{n})$, for each $w \in \mathcal{A}_{\infty}$ and $0 < p \leq 1$, respectively.

\begin{remark} \label{estim fuera de bola}
Let $d = \lfloor n (\frac{\widetilde{q}_w}{p}-1) \rfloor$. If $a(\cdot)$ is a $w-(p, p_0, d)$ atom with $1 < \frac{r_w-1}{r_w} p_0$, then by proceeding as in the estimation of (\ref{Ta}) we find that
\begin{equation*}
| Ta(x) | \leq C\frac{ r^{n+d+1}}{w(B)^{1/p}}  |x - x_0|^{-n-d-1}, \,\,\, x \notin B(x_0, 2r),
\end{equation*}
so
\[
|Ta(x) | \leq C \frac{[M(\chi_{B})(x)]^{\frac{n+d+1}{n}}}{w(B)^{1/p}}, \,\,\, x \notin B(x_0, 2r),
\]
where $M$ is the Hardy-Littlewood maximal operator.
\end{remark}

\begin{lemma} \label{Lpw bound for Ta}
Let $p_0 \in (1, +\infty)$ such that $1 < \frac{r_w -1}{r_w} p_0$. If $T$ is the operator defined in (\ref{T}) and $0 < p \leq 1$, then there exists an universal constant $C>0$ such that $\|Ta\|_{L^{p}_w} \leq C$ for all $w-(p, p_0, d)$ atom $a(\cdot)$.
\end{lemma}

\begin{proof}
Given a $w - (p, p_0, d)$ atom $a(\cdot)$, let $2B=B(x_0, 2r)$, where $B=B(x_0, r)$ is the ball containing the support of $a(\cdot)$. We write
\[
\int_{\mathbb{R}^{n}} |Ta(x)|^{p} w(x) dx = \int_{2B} |Ta(x)|^{p} w(x) dx + \int_{\mathbb{R}^{n} \setminus 2B} |Ta(x)|^{p} w(x) dx = I + II.
\]
Since $T$ is bounded on $L^{p_0}(\mathbb{R}^{n})$ and $w \in RH_{(\frac{p_0}{p})'}$ ($p \leq 1 <  \frac{r_w -1}{r_w} p_0$), the H\"older's inequality applied with $\frac{p_0}{p}$, and the condition $(a2)$ give
\[
I \leq C \| a \|_{p_0}^{p} |B|^{-p/p_0} w(B) = C.
\]
From Remark \ref{estim fuera de bola} and since that $w \in \mathcal{A}_{p\frac{n+d+1}{n}}$, ($p \frac{n+d+1}{n} > \widetilde{q}_{w}$), we get
\[
II \leq w(B)^{-1} \int_{\mathbb{R}^{n} } [M(\chi_{B})(x)]^{p\frac{n+d+1}{n}} w(x) dx \leq C w^{-1}(B) \int_{B} w(x) dx = C,
\]
where the second inequality follows from Theorem \ref{max pesada}. This completes the proof.
\end{proof}

\begin{theorem}
Let $T$ be the operator defined in (\ref{T}). If $w \in \mathcal{A}_{\infty}$ and $0 < p \leq 1$, then $T$ can be extended to an $H^{p}_{w}(\mathbb{R}^{n}) - L^{p}_{w}(\mathbb{R}^{n})$ bounded operator.
\end{theorem}

\begin{proof} The theorem follows from Lemma \ref{Lpw bound for Ta} and Theorem \ref{estim Hpw-Lpw}.
\end{proof}

\subsection{The Riesz potential}

For $0<\alpha <n,$ let $I_{\alpha }$ be the Riesz potential defined by
\begin{equation}
I_{\alpha }f(x)=\int_{\mathbb{R}^{n}} \frac{1}{\left\vert x-y\right\vert ^{n-\alpha }}f(y)dy,
\label{Ia}
\end{equation}%
$f\in L^{s}(\mathbb{R}^{n}),$ $1\leq s < \frac{n}{\alpha}$. A well known result of Sobolev gives the
boundedness of $I_{\alpha }$ from $L^{p}(\mathbb{R}^{n})$ into $L^{q}(%
\mathbb{R}^{n})$ for $1<p<\frac{n}{\alpha }$ and $\frac{1}{q}=\frac{1}{p}-%
\frac{\alpha }{n}.$ In \cite{steinweiss} E. Stein and G. Weiss used the theory of harmonic
functions of several variables to prove that these operators are bounded
from $H^{1}(\mathbb{R}^{n})$ into $L^{\frac{n}{n-\alpha }}(\mathbb{R}^{n}).$
In \cite{T-W}, M. Taibleson and G. Weiss obtained, using the molecular decomposition, the boundedness of the
Riesz potential $I_{\alpha }$ from $H^{p}(\mathbb{R}^{n})$
into $H^{q}(\mathbb{R}^{n}),$ for $0<p \leq 1$ and $\frac{1}{q}=\frac{1}{p}-\frac{%
\alpha }{n}$, independently S. Krantz obtained the same result in \cite{krantz}. We extend these results to the context of the weighted Hardy spaces using the weighted molecular theory developed in Section 3.

\

First we recall the definition of the critical indices for a weight $w$.

\begin{definition} Given a weight $w$, we denote by $\widetilde{q}_{w} = \inf \{ q>1 : w \in \mathcal{A}_q \}$ \textit{the critical index of $w$}, and we denote by $r_{w} = \sup \{ r > 1 : w \in RH_{r} \}$ \textit{the critical index of $w$  for the reverse H\"older condition.}
\end{definition}

\begin{lemma} \label{desig rwps}
Let  $0 < p < 1$. If $w^{1/p} \in \mathcal{A}_{1}$, then $ p \cdot r_{w^{p}} \leq r_{w} \leq r_{w^{p}}$.
\end{lemma}

\begin{proof} The condition $w^{1/p} \in \mathcal{A}_{1}$, with $0 < p < 1$, implies that $w^{p} \in RH_{1/p}$. It is well known that if $w \in RH_{r}$, then $w \in RH_{r + \epsilon}$ for some $\epsilon > 0$, thus $1/p < r_{w^{p}}$. Taking $1/p < t < r_{w^{p}}$ we have that $1 < pt < t$ and $w^{p} \in RH_{t}$, so
\[
\left( \frac{1}{|B|} \int_{B} [w(x)]^{pt} dx \right)^{1/pt} = \left( \frac{1}{|B|} \int_{B} [w^{p}(x)]^{t} dx \right)^{1/pt} \leq C \left( \frac{1}{|B|} \int_{B} w^{p}(x) dx \right)^{1/p} 
\]
\[
\leq C \frac{1}{|B|} \int_{B} w(x) dx,
\]
where the last inequality follows from the Jensen's inequality. This implies that $p \, t < r_w$ for all $t < r_{w^{p}}$, thus $p \cdot r_{w^{p}} \leq r_w$.

By the other hand, since $0 < p < 1$ and $w^{1/p} \in \mathcal{A}_{1}$ we have that $w \in RH_{1/p}$. So $1/p < r_{w}$, taking $1/p < t < r_{w}$ it follows that $1 < pt < t$, and therefore $w \in RH_{pt}$. Then
\[
\left( \frac{1}{|B|} \int_{B} [w^{p}(x)]^{t} dx \right)^{1/t} = \left( \frac{1}{|B|} \int_{B} [w(x)]^{tp} dx \right)^{p/pt} \leq C \left( \frac{1}{|B|} \int_{B} w(x) dx \right)^{p} 
 \]
\[
= C \left( \frac{1}{|B|} \int_{B} [w^{p}(x)]^{1/p} dx \right)^{p} \leq C \frac{1}{|B|} \int_{B} [w^{p}(x)] dx,
\]
where the last inequality it follows from that  $w^{p} \in RH_{1/p}$. So $t < r_{w^{p}}$ for all $t < r_{w}$, this gives $r_{w} \leq r_{w^{p}}$.
\end{proof}

\begin{lemma} \label{ineq rwpqs}
Let  $0 < p < q$. If $w^{q} \in \mathcal{A}_{1}$, then $ p \cdot r_{w^{p}} \leq q \cdot r_{w^{q}}$.
\end{lemma}

\begin{proof} Since $w^{q} \in \mathcal{A}_1$ and $0 < p < q$ we have that $w^{p} \in \mathcal{A}_1$ and $w^{p} \in RH_{q/p}$. thus $q/p < r_{w^{p}}$. Taking $q/p < s < r_{w^{p}}$ we have that $w^{p} \in RH_{s}$ and $1 < p s /q < s$, so 
\[
\left( \frac{1}{|B|} \int_{B} [w^{q}(x)]^{ps/q} dx \right)^{q/ps} = \left( \frac{1}{|B|} \int_{B} [w^{p}(x)]^{s} dx \right)^{q/ps} \leq C \left( \frac{1}{|B|} \int_{B} w^{p}(x) dx \right)^{q/p} 
 \]
\[
\leq C \frac{1}{|B|} \int_{B} w^{q}(x) dx,
\]
where the last inequality follows from the Jensen's inequality. This implies that $\frac{p}{q} \, s  < r_{w^{q}}$ for all $s < r_{w^{p}}$, thus $p \cdot r_{w^{p}} \leq q \cdot r_{w^{q}}$.
\end{proof}

\begin{proposition} \label{Riesz pot on atoms} For $0 < \alpha < n$, let $I_{\alpha}$ be the Riesz potential defined in (\ref{Ia}) and let  $w^{1/s} \in \mathcal{A}_1$, $0 < s <\frac{n}{n + \alpha}$, with $\frac{r_{w} }{r_{w} - 1} < \frac{n}{\alpha}$. 
If  $s \leq p \leq \frac{n}{n + \alpha}$ and $\frac{1}{q} = \frac{1}{p} - \frac{\alpha}{n}$, then $I_{\alpha }a(\cdot)$ is a 
$w^{q} - (q, q_0, \lfloor n (\frac{1}{q} - 1) \rfloor)$ molecule for each $w^{p} - (p, p_0, 2 \lfloor n (\frac{1}{q} - 1) \rfloor + 3 + \lfloor \alpha \rfloor + n )$ atom $a(\cdot)$, where $\frac{r_{w} }{r_{w} - 1} < p_0 < \frac{n}{\alpha}$ and $\frac{1}{q_0} = \frac{1}{p_0} - \frac{\alpha}{n}$.
\end{proposition}

\begin{proof}
 The condition $w^{1/s} \in \mathcal{A}_1$ implies that $w^{p}$ and $w^{q}$ belong to $\mathcal{A}_1$, so $\widetilde{q}_{w^{p}} = \widetilde{q}_{w^{q}} =1$. We observe that $2 \lfloor n (\frac{1}{q} - 1) \rfloor + 3 + \lfloor \alpha \rfloor + n >  \lfloor n (\frac{1}{p} - 1) \rfloor$, thus $a(\cdot)$ is an atom with additional vanishing moments.

Now we shall see that $p \frac{r_{w^{p}}}{r_{w^{p}} - 1} < p_0$ and  $q \frac{r_{w^{q}}}{r_{w^{q}} - 1} < q_0$. The condition  $p \frac{r_{w^{p}}}{r_{w^{p}} - 1} < p_0$ is required in the definition of atom and  $q \frac{r_{w^{q}}}{r_{w^{q}} - 1} < q_0$ in the definition of molecule.

By Lemma \ref{desig rwps} and by hypothesis we have that
\begin{equation}
p \frac{r_{w^{p}}}{r_{w^{p}} - 1} \leq \frac{r_w}{r_w - 1} < p_0.\,\, \label{est1}
\end{equation}
Lemma \ref{ineq rwpqs} and the fact that the function $t \rightarrow \frac{t}{t-1}$ is decreasing on the region $(1, + \infty)$ imply that 
\begin{equation}
\frac{r_{w^{q}}}{r_{w^{q}} - \frac{p}{q}} \leq \frac{r_{w^{p}}}{r_{w^{p}}-1} . \label{est2}
\end{equation}
If $\frac{1}{q_0} = \frac{1}{p_0} - \frac{\alpha}{n}$, from (\ref{est1}) we have
\[
\frac{1}{q_0} <  \frac{r_{w^{p}} - 1}{p \, r_{w^{p}}} - \frac{\alpha}{n},
\]
from (\ref{est2}) we obtain
\[
\frac{1}{q_0} < \frac{r_{w^{q}} - \frac{p}{q}}{p \, r_{w^{q}}} - \frac{\alpha}{n} = \frac{1}{p}\left( 1 - \frac{p}{q \, r_{w^{q}}} \right) - \frac{\alpha}{n} =  \frac{r_{w^{q}}-1}{q \, r_{w^{q}}}.
\]
So $q \frac{r_{w^{q}}}{r_{w^{q}} - 1} < q_0$. 

Now we will show that $I_{\alpha}a(\cdot)$ satisfies the conditions $(m1)$, $(m2)$ and $(m3)$ in the definition of molecule, if $a(\cdot)$ is a  $w^{p} - (p, p_0, 2 \lfloor n (\frac{1}{q} - 1) \rfloor + 3 + \lfloor \alpha \rfloor + n )$ atom.

Since $I_{\alpha}$ is bounded from $L^{p_0}(\mathbb{R}^{n})$ into $L^{q_0}(\mathbb{R}^{n})$ and $w^{p} \in RH_{q/p}$, by 
Lemma \ref{desig RH}, we get
\[
\| I_{\alpha} a\|_{L^{q_0}(B(x_0, 2r))} \leq C \| a \|_{L^{p_0}(\mathbb{R}^{n})} \leq C |B|^{1/p_0} (w^{p}(B))^{-1/p} \leq C |B|^{1/q_0} (w^{q}(B))^{-1/q},
\]
so $I_{\alpha}a(\cdot)$ satisfies $(m1)$.

Let $d = 2 \lfloor n (\frac{1}{q} - 1) \rfloor + 3 + \lfloor \alpha \rfloor + n$, and let $a(\cdot)$ be a $w^{p} - (p, p_0, d)$ atom supported on the ball $B(x_0, r)$. In view of the moment condition of $a(\cdot)$ we obtain
\[
I_{\alpha}a(x) = \int_{B(x_0, r)} \left( |x-y|^{\alpha- n} - q_{d}(x,y) \right) a(y) dy, \,\,\,\,\,\, \textit{for all} \,\,\, x \notin B(x_0, 2r),
\]
where $q_{d}$ is the degree $d$ Taylor polynomial of the function $y \rightarrow |x-y|^{\alpha - n}$ expanded around $x_0$. By the standard estimate of the remainder term of the Taylor expansion, there exists $\xi$ between  $y$ and $x_0$ such that
\[
\left| |x-y|^{\alpha- n} - q_{d}(x,y) \right| \leq C |y - x_0|^{d +1} |x - \xi|^{-n+ \alpha -d -1},
\]
for any $y  \in B(x_0, r)$ and any $x \notin B(x_0, 2r)$, since $|x - \xi| \geq \frac{|x-x_0|}{2}$, we get
\[
\left||x-y|^{\alpha- n} - q_{d}(x,y) \right| \leq C r^{d +1} |x - x_0|^{-n+ \alpha -d -1},
\]
this inequality and the condition $(a2)$ allow us to conclude that
\begin{equation}
|I_{\alpha}a(x) | \leq C \frac{r^{n+d +1}}{(w^{p}(B))^{1/p}} |x-x_0|^{-n + \alpha - d -1}, \,\,\,\,\,\, for \,\, all \,\,\,\, x \notin B(x_0, 2r), \label{Ialfa}
\end{equation}
Lemma \ref{desig RH} and a simple computation gives
\[
|I_{\alpha}a(x) | \leq C (w^{q}(B))^{-1/q} \left(1 +  \frac{|x-x_0|}{r} \right)^{-2n - 2 d_q -3}, \,\,\,\,\,\, for \,\, all \,\,\,\, x \notin B(x_0, 2r),
\]
where $d_q = \lfloor n (\frac{1}{q} - 1) \rfloor$. So $I_{\alpha}a(\cdot)$ satisfies $(m2)$.

Finally, in \cite{T-W} Taibleson and Weiss proved that
\[
\int_{\mathbb{R}^{n}} x^{\beta} I_{\alpha}a(x) dx =0,
\]
for all $0 \leq |\beta| \leq \lfloor n (\frac{1}{q} -1) \rfloor$. This shows that $I_{\alpha}a(\cdot)$ is a $w^{q}$-molecule. The proof of the proposition is therefore concluded.
\end{proof}

\begin{theorem} \label{Hpw-Hqw bound for Riesz} For $0<\alpha <n$, let $I_{\alpha }$ be the Riesz potential defined in (\ref{Ia}). If $w^{1/s} \in \mathcal{A}_1$ with $0 < s <\frac{n}{n + \alpha}$ and $\frac{r_{w} }{r_{w} - 1} < \frac{n}{\alpha}$,  then $I_{\alpha }$ can be extended to an $H^{p}_{w^{p}}\left( \mathbb{R}^{n}\right) - H^{q}_{w^{q}}\left(\mathbb{R}^{n}\right)$ bounded operator for each $s \leq p \leq \frac{n}{n + \alpha}$ and $\frac{1}{q} = \frac{1}{p} - \frac{\alpha}{n}$.
\end{theorem}

\begin{proof} Let $\frac{1}{q} = \frac{1}{p} - \frac{\alpha}{n}$. For the range $p \leq \frac{n}{n + \alpha}$ we have that $p < q \leq 1$. If $p \in [s, \frac{n}{n + \alpha}]$, the condition $w^{1/s} \in \mathcal{A}_1$, $0 < s <\frac{n}{n + \alpha}$, implies that 
$w$, $w^{1/p}$ and $w^{p}$ belong to $\mathcal{A}_1$, so $w^{q} \in \mathcal{A}_1$. Then $\widetilde{q}_{w^{p}} = \widetilde{q}_{w^{q}} = 1$. We put $d_p = \lfloor n(\frac{1}{p} - 1) \rfloor$ and $d_q =  \lfloor n(\frac{1}{q} - 1) \rfloor$. We recall that in the atomic decomposition, we can always choose atoms with additional vanishing moments (see Corollary 2.1.5 pp. 105 in \cite{stein}). This is, if $l$ is any fixed integer with $l > d_p$, then we have an atomic decomposition such that all moments up to order $l$ of our atoms are zero.

For $\frac{r_{w} }{r_{w} - 1} < p_0 < \frac{n}{\alpha}$ we consider $\frac{1}{q_0} = \frac{1}{p_0} - \frac{\alpha}{n}$. We observe that $2 \lfloor n (\frac{1}{q} - 1) \rfloor + 3 + \lfloor \alpha \rfloor + n > \lfloor n (\frac{1}{p} - 1) \rfloor$. Since $w^{1/p} \in \mathcal{A}_1$, from Lemma \ref{desig rwps}, we have $p \frac{r_{w^{p}}}{r_{w^{p}} - 1} \leq  \frac{r_w}{r_w - 1} < p_0$. Thus, given $f \in \widehat{\mathcal{D}}_0$ we can write $f = \sum_{j} \lambda_j a_j$, where $a_j$ are $w^{p} - (p, p_0, 2 \lfloor n (\frac{1}{q} - 1) \rfloor + 3 + \lfloor \alpha \rfloor + n  )$ atoms, $\sum_j  |\lambda_j |^{p} \lesssim \| f \|_{H^{p}_{w^{p}}}^{p}$ and the series converges in 
$L^{p_0}(\mathbb{R}^{n})$. Since $I_{\alpha}$ is a 
$L^{p_0}(\mathbb{R}^{n}) - L^{q_0}(\mathbb{R}^{n})$ bounded operator it follows that $I_{\alpha} f = \sum_{j} \lambda_j I_{\alpha} a_j$ in $L^{q_0}(\mathbb{R}^{n})$ and therefore in $\mathcal{S}'(\mathbb{R}^{n})$. By Proposition \ref{Riesz pot on atoms}, we have that the operator $I_{\alpha}$ maps $w^{p} - (p, p_0, 2 \lfloor n (\frac{1}{q} - 1) \rfloor + 3 + \lfloor \alpha \rfloor + n  ))$ atoms $a(\cdot)$ to $w^{q} - (q, q_0, d_q)$ molecules $I_{\alpha}a(\cdot)$, applying Theorem \ref{Hp estim mol} we get
\[
\| I_{\alpha} f \|_{H^{q}_{w^{q}}}^{q} \lesssim \sum_{j} |\lambda_{j}|^{q} \lesssim \left( \sum_{j} |\lambda_j |^{p} \right)^{q/p} \lesssim \|f \|_{H^{p}_{w^{p}}}^{q},
\]
for all $f \in \widehat{\mathcal{D}}_0$, so the theorem follows from the density of  $\widehat{\mathcal{D}}_0$ in $H^{p}_{w^{p}}(\mathbb{R}^{n})$.
\end{proof}

 For $\frac{n}{n + \alpha} < p \leq 1$, we have that $1 < q \leq \frac{n}{n - \alpha}$. For this range of $q$'s the space $H^{q}_{w}$ can be identify with the space $L^{q}_{w}$. The following theorem contains this range of $p$'s.

\begin{theorem} \label{Hpw-Lqw cota for Riesz} For $0<\alpha <n$, let $I_{\alpha }$ be the Riesz potential defined in (\ref{Ia}). If $w^{\frac{n}{(n-\alpha) \, s}} \in \mathcal{A}_{1}$ with $0 < s < 1$ and $\frac{r_{w} }{r_{w} - 1} < \frac{n}{\alpha}$, 
then $I_{\alpha }$ can be extended to an $H^{p}_{w^{p}}\left( \mathbb{R}^{n}\right) - L^{q}_{w^{q}}\left(\mathbb{R}^{n}\right)$ bounded operator for each $s \leq p \leq 1$ and $\frac{1}{q}=\frac{1}{p}-\frac{\alpha }{n}$.
\end{theorem}

\begin{proof}
The condition $w^{n/(n - \alpha)s} \in \mathcal{A}_1$, $0 < s < 1 < \frac{n}{n - \alpha}$, implies that 
$w$, $w^{1/p}$, $w^{p}$ and $w^{q}$ belong to $\mathcal{A}_1$, for all $s \leq p \leq 1$ and $\frac{1}{q} = \frac{1}{p} - \frac{\alpha}{n}$.

We take $p_0$ such that $\frac{r_w}{r_w - 1} < p_0 < \frac{n}{\alpha}$, from lemma \ref{desig rwps}, we have that $p \frac{r_{w^{p}}}{r_{w^{p}} -1} \leq \frac{r_w}{r_w -1} < p_0$. Given $f \in \widehat{\mathcal{D}}_{0}$ we can write $f = \sum \lambda_j a_j$ where the $a_j$'s are $w^{p} - (p, p_0, d)$ atoms, the scalars $\lambda_j$ satisfies $\sum_{j} |\lambda_j |^{p} \lesssim \| f \|_{H^{p}_{w^{p}}}^{p}$  and the series converges in $L^{p_0}(\mathbb{R}^{n})$. For $\frac{1}{q_0} = \frac{1}{p_0} - \frac{\alpha}{n}$, $I_{\alpha}$ is a bounded operator from $L^{p_0}(\mathbb{R}^{n})$ into $L^{q_0}(\mathbb{R}^{n})$, since $f = \sum_j \lambda_j a_j$ in $L^{p_0}(\mathbb{R}^{n})$, we have that
\begin{equation}
|I_{\alpha}f(x)| \leq \sum_{j} |\lambda_j| |I_{\alpha}a_j(x)|, \,\,\,\,\, \textit{a.e.} x \in \mathbb{R}^{n}. \label{puntual}
\end{equation}
If $\|I_{\alpha} a_j \|_{L^{q}_{w^{q}}} \leq C$, with $C$ independent of the $ w^{p} - (p, p_0, d)$ atom $a_j(\cdot)$, then (\ref{puntual}) allows us to obtain
\[
\|I_{\alpha} f \|_{L^{q}_{w^{q}}} \leq C \left( \sum_{j} |\lambda_j|^{\min\{1, q \}} \right)^{\frac{1}{\min\{1, q \}}} \leq C \left( \sum_{j} |\lambda_j |^{p} \right)^{1/p} \lesssim \| f \|_{H^{p}_{w^{p}}},
\] 
for all $f \in \widehat{\mathcal{D}}_0$, so the theorem follows from the density of  $\widehat{\mathcal{D}}_0$ in $H^{p}_{w^{p}}(\mathbb{R}^{n})$.

\

To conclude the proof we will prove that there exists $C > 0$ such that 
\begin{equation}
\|I_{\alpha} a \|_{L^{q}_{w^{q}}} \leq C, \,\,\,\, \textit{for all} \,\, w^{p} - (p, p_0, d) \,\, \textit{atom} \,\, a(\cdot). \label{uniform estimate}
\end{equation}
To prove (\ref{uniform estimate}), let $2B=B(x_0, 2r)$, where $B=B(x_0, r)$ is the ball containing the support of the atom $a(\cdot)$. So
\[
\int_{\mathbb{R}^{n}} |I_{\alpha}a(x)|^{q} w^{q}(x) dx = \int_{2B} |I_{\alpha}a(x)|^{q} w^{q}(x) dx + \int_{\mathbb{R}^{n} \setminus 2B} |I_{\alpha}a(x)|^{q} w^{q}(x) dx 
\]
To estimate the first term in the right-side of this equality, we apply H\"older's inequality with $\frac{q_0}{q}$ and use that $w^{q} \in RH_{(\frac{q_0}{q})'}$ ($q_0 > q \frac{r_{w^{q}}}{r_{w^{q}}-1})$, thus
\[
\int_{2B} |I_{\alpha}a(x)|^{q} w^{q}(x) dx \leq \|I_{\alpha} a \|_{L^{q_0}}^{q} \left( \int_{2B} [w^{q}(x)]^{(\frac{q_0}{q})'} dx \right)^{1/(\frac{q_0}{q})'}
\]
\[
\leq C |B|^{q/p_0} (w^{p}(B))^{-q/p} |2B|^{1/(\frac{q_0}{q})'} \left( \frac{1}{|2B|} \int_{2B} w^{q}(x) dx \right)
\]
\[
\leq C |B|^{q\alpha/n}  (w^{p}(B))^{-q/p} w^{q}(B).
\]
Lemma \ref{desig RH} gives
\begin{equation}
\int_{2B} |I_{\alpha}a(x)|^{q} w^{q}(x) dx \leq C. \label{Ialfa 2}
\end{equation}
From (\ref{Ialfa}), taking there $d= \lfloor n(\frac{1}{p}-1) \rfloor$, we obtain
\[
|I_{\alpha}a(x)| \leq C (w^{p}(B))^{-1/p}\left[ M_{\frac{\alpha n}{n+d+1}}(\chi_{B}) (x) \right]^{\frac{n+d+1}{n}}, \,\,\,\,\,\, \textit{for all} \,\, x \notin B(x_0, 2r).
\]
So
\begin{equation}
\int_{\mathbb{R}^{n} \setminus 2B} |I_{\alpha}a(x)|^{q} w^{q}(x) dx \leq C (w^{p}(B))^{-q/p} \int_{\mathbb{R}^{n}} \left[ M_{\frac{\alpha n}{n+d+1}}(\chi_{B}) (x) \right]^{q\frac{n+d+1}{n}} w^{q}(x) dx, \label{Ialfa 3}
\end{equation}
Since  $d= \lfloor n(\frac{1}{p}-1) \rfloor$, we have $q \frac{n+d+1}{n} > 1$. We write $\widetilde{q} = q \frac{n+d+1}{n}$ and let $\frac{1}{\widetilde{p}} = \frac{1}{\widetilde{q}} + \frac{\alpha}{n+d+1}$, so $\frac{\widetilde{p}}{\widetilde{q}} = \frac{p}{q}$ and
$w^{q/{\widetilde{q}}} \in \mathcal{A}_{\widetilde{p}, \widetilde{q}}$ (see Remark \ref{A_1 en A_pq}). From Theorem \ref{fract max pesada} we obtain
\[
\int_{\mathbb{R}^{n}} \left[ M_{\frac{\alpha n}{n+d+1}}(\chi_{B}) (x) \right]^{q\frac{n+d+1}{n}} w^{q}(x) dx \leq 
C \left( \int_{\mathbb{R}^{n}} \chi_{B}(x) w^{p}(x) dx \right)^{q/p} = C (w^{p}(B))^{q/p}.
\]
This inequality and (\ref{Ialfa 3}) give
\begin{equation}
\int_{\mathbb{R}^{n} \setminus 2B} |I_{\alpha}a(x)|^{q} w^{q}(x) dx \leq C. \label{Ialfa 4}
\end{equation}
Finally, (\ref{Ialfa 2}) and (\ref{Ialfa 4}) allow us to obtain (\ref{uniform estimate}). This completes the proof.
\end{proof} 

To finish, we recover the classical result obtained by Taibleson and Weiss in \cite{T-W}.

\begin{corollary}
For $0 < \alpha < n$, let $I_{\alpha}$ be the Riesz potential defined in (\ref{Ia}). If $0 < p \leq 1$ and $\frac{1}{q} = \frac{1}{p} - \frac{\alpha}{n}$, then $I_{\alpha}$ can be extended to an $H^{p}\left( \mathbb{R}^{n}\right) - H^{q}\left(\mathbb{R}^{n}\right)$ bounded operator.
\end{corollary}

\begin{proof} If $w(x) \equiv 1$, then $r_{w} = + \infty$ and therefore $\frac{r_{w}}{r_w - 1} = 1$. Applying the theorems \ref{Hpw-Hqw bound for Riesz} and \ref{Hpw-Lqw cota for Riesz}, with $w \equiv 1$, the corollary follows.
\end{proof}



\begin{thebibliography}{99}

\bibitem{bownik} M. Bownik, Boundedness of operators on Hardy spaces via atomic decompositions. \textit{Proc. Amer. Math. Soc.}, 133, 3535-3542, (2005).

\bibitem{coifman} R. Coifman, A real variable characterization of $H^{p}$, \textit{Studia Math.}, 51, 269-274, (1974).

\bibitem{uribe} D. Cruz-Uribe, K. Moen,  and H. Van Nguyen, The boundedness of multilinear calder\'o-Zygmund operators on weighted and variable Hardy spaces. Preprint 2017, arXiv:1708.07195.

\bibitem{cuerva} J. Garc\'{\i}a-Cuerva, Weighted $H^{p}$ spaces, \textit{Dissertations Math.}, 162, 1-63, (1979).

\bibitem{F-S} C. Fefferman and E. M. Stein, $H^{p}$ spaces of several variables, \textit{Acta Math.} 129 (3-4), 137-193, (1972).

\bibitem{grafakos} L. Grafakos, Classical Fourier Analysis, 3rd edition, Graduate Texts in Mathematics, 249, Springer New York, (2014).

\bibitem{krantz} S. Krantz, Fractional integration on Hardy spaces, \textit{Studia Math.}, (73), 87-94, (1982). 

\bibitem{latter} R. Latter, A characterization of $H^{p}(\mathbb{R}^{n})$ in terms of atoms, \textit{Studia Math.}, 62, 93-101, (1978).

\bibitem{lee} M. -Y. Lee, C. -C. Lin, The molecular characterization of weighted  Hardy spaces, \textit{Journal of Funct. Analysis,} 188, 442 - 460, (2002).

\bibitem{li} X. Li and L. Peng, The molecular characterization of weighted Hardy spaces, \textit{Science in China} (Series A), vol. 44 (2), 201-211, (2001).

\bibitem{Muck} B. Muckenhoupt, Weighted norm inequalities for the Hardy maximal function, \textit{Trans. of the Amer. Math. Soc.}, Vol 165, 207-226, (1972).

\bibitem{Muck2} B. Muckenhoupt and R. L. Wheeden, Weighted norm inequalities for fractional integrals, \textit{Trans. of the Amer. Math. Soc.}, Vol 192, 261-274, (1974).

\bibitem{nakai} E. Nakai and Y. Sawano, Hardy spaces with variable exponents and generalized Campanato spaces, \textit{Journal of Functional Analysis}, 262, 3665-3748, (2012).

\bibitem{rocha} P. Rocha, A note on Hardy spaces and bounded linear operators, \textit{Georgian Math. J.}, 25(1), 73-76, (2018).

\bibitem{elias}  E. M. Stein, Singular Integrals and Differentiability Properties of Functions, Princeton Univ. Press, Princeton, NJ, (1970).

\bibitem{stein} E. M. Stein, Harmonic Analysis, Real-Variable Methods, Orthogonality, and Oscillatory integrals, Princeton Univ. Press, Princeton, NJ, (1993).

\bibitem{steinweiss}  E. M. Stein and G. Weiss, On the theory of harmonic functions of several variables I: The theory of $H^{p}$ spaces, \textit{Acta Math.}, 103, 25-62, (1960).

\bibitem{tor} J. O. Str\"omberg and A. Torchinsky, Weighted Hardy spaces, Lecture Notes in Mathematics, vol 131, Springer-Verlag, Berl\'{\i}n, (1989).

\bibitem{wheeden}  J. O. Str\"omberg and R. L. Wheeden, Fractional integrals on weighted $H^{p}$ and $L^{p}$ spaces, \textit{Trans. Amer. Math. Soc.}, 287, 293-321, (1985).

\bibitem{T-W} M. H. Taibleson and G. Weiss, The molecular characterization of certain Hardy spaces,\textit{ Ast\'{e}risque 77,} 67-149, (1980).

\bibitem{dachun} D. Yang and Y. Zhou, A boundedness Criterian via atoms for linear operators in Hardy spaces, \textit{Constr. Approx.}, 29, 207-218, (2009).

\end{thebibliography}
\end{document}